\documentclass[12pt]{amsart}
\usepackage{amsthm,amsmath,amsxtra,amscd,amssymb,color}

\usepackage{verbatim}

\newcommand{\de}{\delta}

\newcommand{\CC}{{\mathbb{C}}}

\newcommand{\PP}{{\mathbb{P}}}

\newcommand{\ZZ}{{\mathbb{Z}}}

\newcommand{\calO}{{\mathcal O}}
\newcommand{\calG}{{\mathcal G}}
\newcommand{\calH}{{\mathcal H}}

\newcommand{\calA}{{\mathcal A}}

\newcommand{\calJ}{{\mathcal J}}

\newcommand{\m}{\mathfrak{m}}
\newcommand{\n}{\mathfrak{n}}
\newcommand{\q}{\mathfrak{q}}
\newcommand{\ee}{\mathfrak{e}}

\DeclareMathOperator{\red}{red}

\newcommand{\op}{\operatorname}

\newcommand{\Sing}{\op{Sing}}

\newcommand{\mult}{\op{mult}}
\newcommand{\len}{\op{length}}

\newcommand\codim{{\rm codim}}

\theoremstyle{plain}
\newtheorem{thm}{Theorem}[section]
\newtheorem{lm}[thm]{Lemma}
\newtheorem{prop}[thm]{Proposition}
\newtheorem{cor}[thm]{Corollary}

\newtheorem{conj}[thm]{Conjecture}

\newtheorem{question}[thm]{Question}

\theoremstyle{definition}
\newtheorem{example}[thm]{Example}
\newtheorem{df}[thm]{Definition}
\newtheorem{rem}[thm]{Remark}

\begin{document}
\title{The degree of the Gauss map of the theta divisor}
\author{Giulio Codogni}
\address{Universit\`a Roma Tre, Dipartimento di Matematica e Fisica, Largo San Leonardo Murialdo I-00146 Roma, Italy}
\email{codogni@mat.uniroma3.it}
\author{Samuel Grushevsky}
\address{Mathematics Department, Stony Brook University, Stony Brook, NY 11794-3651, USA}
\email{sam@math.stonybrook.edu}
\thanks{The first author is funded by the FIRB 2012 ``Moduli spaces and their applications" and the ``ERC StG 307119 - Stab AGDG". Research of the second author is supported in part by National Science Foundation under the grant DMS-15-01265, and by a Simons Fellowship in Mathematics (Simons Foundation grant \#341858 to Samuel Grushevsky)}
\author{Edoardo Sernesi}
\address{Universit\`a Roma Tre, Dipartimento di Matematica e Fisica, Largo San Leonardo Murialdo I-00146 Roma, Italy}
\email{sernesi@mat.uniroma3.it}

\begin{abstract}
We study the degree of the Gauss map of the theta divisor of principally polarised complex abelian varieties. Thanks to this analysis, we obtain a bound on the multiplicity of the theta divisor along irreducible components of its singular locus. We spell out this bound in several examples, and we use it to understand the local structure of isolated singular points. We further define a stratification of the moduli space of ppav's by the degree of the Gauss map. In dimension four, we show that this stratification gives a weak solution of the Schottky problem, and we conjecture that this is true in any dimension.
\end{abstract}
\maketitle

\section{Introduction}
Let $(A,\Theta)$ be a complex $g$-dimensional principally polarised abelian variety (ppav), where, by abuse of notation, we write $\Theta$ both for an actual symmetric divisor of $A$ and for its first Chern class. The Gauss map
$$\calG:\Theta\dashrightarrow\PP^{g-1}$$
is the rational map given by the complete linear system $L:=\calO(\Theta)|_\Theta$. If $(A,\Theta)$ is indecomposable (not a product of lower-dimensional ppav's), then $\calG$ is a rational dominant generically finite map, and we are interested in its degree. Since a basis of $H^0(\Theta,L)$ is given by the partial derivatives of the theta function restricted to the theta divisor, the scheme-theoretic base locus of the Gauss map is equal to the singular locus $\Sing(\Theta)$, endowed with the scheme structure defined by the partial derivatives of the theta function.

Given an irreducible component $V$ of $\Sing(\Theta)$, we denote by $V_{red}$ the reduced scheme structure on $V$, and then denote $\mult_{V_{red}}\Theta$ the vanishing multiplicity of the theta function along  $V_{red}$. We denote by $\deg_V\Theta$ the intersection number $\Theta^d\cdot V$  on $A$, where $d=\dim V$; note that  $\deg_V\Theta\ge 1$ for any $V$, since the theta divisor is ample. Our main result relates the multiplicity of the singularities of the theta divisor with the degree of the Gauss map.
\begin{thm}\label{thm:multbound}
Let $(A,\Theta)$ be an indecomposable principally polarized abelian variety of dimension $g\ge 3$.
Let $\{V_i\}_{i\in I}$ be the set of irreducible components of  $\Sing(\Theta)$; denote $d_i:=\dim V_i$, and $m_i:=\mult_{V_{i,red}}\Theta$. Then the following inequalities hold:
$$
 \sum_{i\in I} m_i(m_i-1)^{g-d_i-1}\deg_{V_{i,\red}}\Theta\leq g!-\deg \calG     \le g! -4,
$$
where $\deg \calG$ is the degree of the Gauss map.
\end{thm}
In Section~\ref{section:multi} we discuss the improvements of this result in various cases, by using the known lower bounds on the degree of the theta divisor on subschemes of $A$. In particular, as an easy corollary we obtain a special case of a well-known result of Ein and Lazarsfeld~\cite{EinLaz}.
\begin{cor}\label{smooth_cod_1}
For any indecomposable ppav $(A,\Theta)$ such that the algebraic cohomology group $HA^{2,2}(A)=\ZZ$, the divisor~$\Theta$ is smooth in codimension~$1$.
\end{cor}

\begin{rem}
For the case of an isolated singular point $z\in\Sing(\Theta)$ of multiplicity $m$, the  bound of Theorem~\ref{thm:multbound}   gives
$m(m-1)^{g-1}\le g!-4$, which asymptotically for $g\to\infty$ behaves like $m\le g/e$, where $e$ is the Euler number.

Recall that the inequality $\mult_z\Theta\le g$ at any point of any ppav is a famous result of Koll\'ar~\cite{kollar}, while Smith and Varley~\cite{smva} prove that there exists a point $z$ such that $\mult_z\Theta=g$, if and only if the ppav is the product of $g$ elliptic curves. Conjecturally (see the discussion in~\cite{Sam}) for indecomposable ppav's the multiplicity of the theta divisor at any point is at most $\lfloor \frac{g+1}{2}\rfloor$ --- this is the maximum possible multiplicity for Jacobians of smooth curves. The bound we get for multiplicity of isolated singular points is thus asymptotically better than this conjecture.
\end{rem}
Since the multiplicity $m_i\ge 2$, we also get the following
\begin{cor}\label{iso_points}
The number of isolated singular points of the theta divisor is at most $(g!-4)/2$.
\end{cor}
\begin{rem}
Recall that by definition a {\em vanishing theta-null} for a ppav is an even two-torsion point contained in the theta divisor. Conjecturally for an indecomposable ppav the number of vanishing theta-nulls is less than $2^{g-1}(2^g+1)-3^g$, and this conjecture was recently investigated in detail by  Auffarth, Pirola, Salvati Manni~\cite{aufpism}. The corollary above provides a better bound than this conjecture for all ppav's for genus up to $8$ (this is the range in which $(g!-4)/2<2^{g-1}(2^g+1)-3^g$ ) such that every vanishing theta-null is an isolated point of $\Sing(\Theta)$.
\end{rem}
We use the machinery of Vogel's $v$-cycles to obtain the proof of our result. As we only need to apply this machinery in a specific case, we give a direct elementary construction of $v$-cycles, and a self-contained argument for our results, not relying on the general $v$-cycle literature.

\subsection*{Related results and approaches}
Using the results of~\cite{Laz} and~\cite{Nakamaye} one can obtain a slightly weaker version of the bound on the multiplicity of singularities claimed in Theorem~\ref{thm:multbound}.  More recently, Musta\c t\u a and Popa applied their newly developed theory of Hodge ideals to also get a slightly weaker multiplicity bound, see~\cite[Sec.~29]{MuPo}. Neither of these works is directly related to the degree of the Gauss map. R.~Varley explained to us that Corollary~\ref{iso_points} can be obtained by extending the techniques from~\cite{smva} and utilizing the machinery developed in~\cite[section~4.4 and  chapter~11]{fultonbook}.

\bigskip
In the spirit of the Andreotti-Mayer loci defined as
$$
 N_k^{(g)}:=\lbrace (A,\Theta)\in\calA_g\,|\, \dim\Sing(\Theta)\ge k\rbrace\,,
$$
we define the {\em Gauss loci}
$$
 G_d^{(g)}:=\lbrace (A,\Theta)\in\calA_g\,|\, \deg\left(\calG:\Theta\dashrightarrow\PP^{g-1}\right)\le d\rbrace.
$$
In this spirit, we conjecture that the degree of the Gauss map gives a weak solution to the Schottky problem:
\begin{conj}\label{conj}
The Gauss loci  $G_d^{(g)}$ are closed for any $d$ and $g$. The closure of the locus of Jacobians of smooth curves $\overline\calJ_g$ is an irreducible component of $G_{d'}^{(g)}$, where $d'=\binom{2g-2}{g-1}$. The closure of the locus of Jacobians of smooth hyperelliptic curves $\overline\calH_g$ is an irreducible component of $G_{2^{g-1}}^{(g)}$.
\end{conj}
In making this conjecture, we recall that the degree of the Gauss map is known for all Jacobians of smooth curves: it is equal to $2^{g-1}$ if the curve is hyperelliptic, and to $\binom{2g-2}{g-1}$ if the curve is not hyperelliptic, see~\cite[Proof of  Prop.~10]{torelli}.

In Proposition~\ref{prop:genus4close}, we prove Conjecture~\ref{conj} for $g=4$ (after easily verifying it for any $g<4$). We also show that all relevant Gauss loci are distinct in genus 4, i.e.~that for any  $2\le d\le 12$ there exists an abelian fourfold such that the degree of its Gauss map is equal to $2d$.

Already when $g=5$, we expect not all the even degree Gauss loci $G_{2d}^{(5)}$ for $2\le d\le 60$ to be distinct, and in general the following natural question is completely open:
\begin{question}
For a given $g\geq 5$, what is the set of possible degrees of the Gauss map for $g$-dimensional ppav's? Equivalently, for which $d$ is $G_{2d}^{(g)}\setminus G_{2d-2}^{(g)}$ non-empty?
\end{question}

{\small \subsection*{Acknowledgments}
We thank P. Aluffi, R. Lazarsfeld, R. Salvati Manni, A. Verra and F. Viviani for many stimulating conversations, and R. Auffarth, T. Kr{\"a}mer, S. Schreider and R. Varley for valuable comments on a preliminary version of this paper.

%%%%%%%%%%%%%%%%%%%%%%%%%%%%%%%%%%%%%%

\section{Generalities about the Gauss map and multiplicities of singularities}
Our paper is focused on studying the degree and the geometry of the Gauss map for abelian varieties. We first recall the general setup for working with generically finite rational maps to projective spaces.

\smallskip
Let $X$ be an $n$-dimensional complex projective variety (we implicitly assume all varieties in this paper to be irreducible, unless stated otherwise), $L$ a line bundle on $X$ and $W\subseteq H^0(X,L)$ a vector subspace such that the rational map
$$
 f=|W|:X\dashrightarrow \PP W^{\vee}
$$
is generically finite and dominant. Let $B$ be the scheme-theoretic base locus of $f$. Our focus will be the discrepancy of $f$:
\begin{df}[Discrepancy]
In the above setup, the discrepancy $\delta$ of $f$ is
$$
  \delta:=\deg_XL-\deg f.
$$
\end{df}
If $f$ is regular, the discrepancy $\delta$ is equal to zero. In general, naively, $\delta$ tells us how much of the degree of $L$ is absorbed by the base locus $B$. Note that the discrepancy is not always positive: the presence of the base locus could even increase the degree, as shown in Example~\ref{example:negative}. We will show that when $L$ is ample and the base locus is not empty, the discrepancy is strictly positive.

\smallskip
The following formula for the discrepancy in terms of Segre classes is given in~\cite[Prop.~4.4]{fultonbook}:
$$
 \delta=\int_B(1+c_1(L))^n\cap s(B,X).
$$
For our purposes in this paper  we prefer to avoid Segre classes and use the language of Vogel's cycles, called $v$-cycles for short. The theory of $v$-cycles is fully developed and described in detail in~\cite{flennerbook}, other references are~\cite{vogel} and~\cite{vangastel}. However, as the applications we need in this paper are limited, for our purposes a limited version of the theory suffices. We thus prefer to give an entirely self-contained exposition of the machinery we use --- avoiding most technicalities, and not having to rely on the literature on the subject.

\medskip
In the setup above, we will define effective cycles $V^j$ with support contained in $B$ and the so-called residual  cycles $R^j$ such that the support of none of their irreducible components is contained in $B$, with $j=0,\dots, n$. Both $V^j$ and $R^j$ are equidimensional of dimension $n-j$, unless they are empty.

These cycles are defined inductively --- according to Vogel's intersection algorithm. We let $V^0:=\emptyset$, and let $R^0:=X$. For any $j>0$ assume we have already defined effective cycles $V^{j-1}$ and $R^{j-1}$ of pure dimension $n-j+1$. Since none of the irreducible components of $R^{j-1}$ have support contained in $B$, the zero locus $D_j\subset X$ of a generic section $s_j\in W$ does not contain any irreducible component of $R^{j-1}$. We thus write $D_j\cap R^{j-1}=\sum_k a_k E_k$, where the $E_k$ are reduced and irreducible Cartier divisors on $R^{j-1}$. We order the divisors $E_k$ in such a way that $E_k$ is supported within $B$ for $k=1,\dots , t$ and is not supported within $B$ for $k>t$. Then we let
$$
  V^j:=\sum_{k=1}^t a_kE_k \quad \textrm{and} \quad R^j:=\sum_{k>t}a_kE_k.
$$
The cycles $R^j$ and $V^j$ are effective of pure dimension $n-j$. The cycle $V^j$ can possibly be empty; on the other hand, for $D_j$ general, the cycle $R^j$ is non-empty.
Since
$$
  D_j\cap R^{j-1}=V^j+R^j\,,
$$
and since $D_j$ is the zero locus of a section of $L$, it follows that
\begin{equation}\label{Rii-1}
  \deg_{R^j} L=\deg_{R^{j-1}}L-\deg_{V^j} L.
\end{equation}
\begin{thm}[Formula for the discrepancy]\label{main}
In the setup above, the discrepancy is given by
$$
\delta=\sum_{j=1}^n \deg_{V^j} L\,.
$$
\end{thm}
\begin{proof}
For $0\leq j\leq n$, consider the restricted morphisms
$$
  f_j:=f|_{R^j}\colon R^j \dashrightarrow \PP H_j\,,
$$
where $H_j$ is the linear subspace of $W^{\vee}$ that is the common zero locus of the sections $s_1,\ldots,s_j\in W$. Since $R^j$ is the closure of $f^{-1}(\PP H_j)$ in $X$, we have
$$
  \deg f=\deg f_j
$$
for any $j$. We now take $j=n$. Since $R^n$ is zero-dimensional and it is supported outside $B$, the map
$$
  f_n\colon R^n \to \PP H_n
$$
is a regular morphism, and thus
$$
  \deg f=\deg f_n=\deg_{R^n}L.
$$
To compute $\deg_{R^n}L$, we apply formula~\eqref{Rii-1} $n$ times to obtain
$$
\deg_{R^n}L=\deg_{R^0}L-\sum_{j=1}^n\deg_{V^j}L\,.
$$
The theorem follows from observing that $\deg_{R^0}L=\deg_XL$.
\end{proof}
\begin{example}\label{example:1}
Let $\ell$ be a line in $\PP^3$ and $W\subset H^0(\PP^3,\calO(3))$ define a generic $3$-dimensional linear system of cubics containing $\ell$. This linear system gives a rational map with finite fibres
$$
 f=|W|:\PP^3\dashrightarrow \PP W^{\vee}\,.
$$
Let us run Vogel's algorithm to describe the $v$-cycles and compute the discrepancy. We thus pick $4$ generic cubics $D_1,\dots , D_4\in W$, and have $R^1=D_1$ and $V^1=\varnothing$. The intersection $D_2\cap R^1$ is a reducible curve $ C\cup \ell$; thus $R^2=C$ and $V^2=\ell$. The intersection $C\cap D_3$, is equal to the union of  $C\cap \ell$ and $r$ reduced points not contained in $\ell$, and thus $V^3=C\cap \ell$ and $R^3$ consists of the other $r$ reduced points. To describe $C\cap \ell$, notice that the arithmetic genus of $C\cup \ell$ is $10$. The curve $C$ is a component of a complete intersection; since we know the degree of all components of this complete intersection, we can show that the genus of $C$ is $7$ using a standard formula, see for example~\cite[Proposition 11.6]{Sernesi}. We conclude that $C\cap \ell$ consists of $4$ reduced points.

The discrepancy of the morphism associated to $W$ is
$$
\delta=\deg_{\calO(3)}\ell+\deg_{\calO(3)}V^3=3+4=7\,.
$$
Hence the rational map given by $W$ has degree $\deg_{\PP^3}\calO(3)-\delta=27-7=20$.

We can also compute this discrepancy using Segre classes. Notice that $\ell$ is regularly embedded in $\mathbb{P}^3$, so its Segre class is just the inverse of the Chern class of the normal bundle. Let $p$ be the class of a point in $\ell$. We have
$$
\delta=\int_{\ell}(1+3p)^3(1-2p)=\int_{\ell}(1+9p-2p)=9-2=7\,.
$$
It is worth remarking that the Segre classes formula expresses the discrepancy as a difference, whereas the $v$-cycles formula expresses the discrepancy as a sum of positive contributions.
\end{example}
We now modify slightly the previous example to obtain a map with negative discrepancy (which is of course impossible for an ample line bundle $L$). The highlight of the following example is that the discrepancy is negative because the degree of $L$ on the $v$-cycle $V^2$ is negative; in particular, $L$ is not nef.

\begin{example}[Negative discrepancy]\label{example:negative}
Keeping the notation of Example~\ref{example:1}, we consider the blow-up $\pi\colon X\to \PP^3$ of $\PP^3$ at generic points $p_1, \dots , p_k$ on the line $\ell$; in particular, we assume that these points are not contained in the curve $C$ described in Example~\ref{example:1}. Let $E_1,\ldots,E_k$ be the exceptional divisors and let $E=\sum E_i$. On $X$, we take as line bundle $L=\pi^*\calO(3)-E$, and as linear system we take the proper transform $\widetilde{W}$ of $W$; denote by $\tilde{\ell}$ the proper transform of $\ell$.

The $v$-cycles are the proper transform of the $v$-cycles of the previous example, so $V^2=\tilde{\ell}$, and $V^3$ consists again of $4$ reduced points on $\tilde{\ell}$. The discrepancy is now
$$
\delta=\deg_L\tilde{\ell}+\deg_LV^3=(3-k)+4=7-k,
$$
which is negative for $k>7$. The degree of the map is $$\deg_XL-\delta=(27-k)-(7-k)=20\,.$$ As expected, this degree agrees with the degree of the map described in Example~\ref{example:1}; the reason is that the two morphisms coincide outside the exceptional locus of the blow-up.
\end{example}
Before giving an application of Theorem~\ref{main}, we need to recall the notion of  multiplicity. Let $Z$ be an irreducible subscheme of an irreducible scheme $X$; the  \emph{multiplicity $\mult_ZX$ of $X$ along $Z$} is defined using the notion of Samuel multiplicity as follows. Let $(R,\m)$ be the local ring of $X$ at (the generic point of) $Z_{\red}$. In this ring, $Z$ is defined by an $\m$-primary ideal $\q$. For $t$ sufficiently large, the Hilbert function $h(t):=\len(R/\q^t)$ is a polynomial in $t$. We normalise this polynomial by multiplying it by $\deg(h(t))!=\codim(Z,X)!\,$, and the \emph{Samuel multiplicity} $\ee(\q,R)$ is then defined to be the leading term of the normalised Hilbert polynomial.

Then one defines
\[
  \mult_Z X := \ee(\q,R)\,.
\]
The following commutative algebra result, roughly speaking, reduces the computation of the Samuel multiplicity to the case of local complete intersections.
\begin{prop}\label{CAtheorem}
Let $(R,\m)$ be a $d$-dimensional Cohen-Macaulay local ring with infinite residue field $K$. Let $\q$ be an $\m$-primary ideal of $R$. Let $t_1,\dots , t_m$ be a set of generators of $\q$, and let $s_1,\dots, s_d$ be generic linear combinations of the $t_i$. Then the ideal $I:=(s_1,\dots, s_d)\subseteq \q$ computes the Samuel multiplicity of $\q$, that is
$$
\ee(\q,R)=\ee(I,R) = \ell(R/I).
$$
\end{prop}
\begin{proof}  From~\cite[Thms. 14.13 and 14.14 p. 112]{matsumura} it follows that $\ee(\q,R)=\ee(I,R)$. Since $R$ is Cohen-Macaulay, $\underline s:=s_1,\dots, s_d$ is a regular $R$-sequence; thus
\[
\ell(R/I) = \ell(H_0(\underline s, R)) = \sum_i(-1)^i \ell(H_i(\underline s, R)) = \ee(I,R)
\]
by ~\cite[p. 109]{matsumura}.
\end{proof}
Before applying this result, let us recall that by definition the class of a closed irreducible subscheme $Z\subset X$ in the Chow group of $X$ is $$ [Z]=\ell(\calO_{X,Z_{\red}}/\mathcal{I}_Z)\cdot[Z_{\red}]\,,$$
as explained for example in~\cite[Section 21.1.1]{Voisin}.
\begin{cor}[Bound on the discrepancy]\label{bound_disc}
In the setup above, assume that $X$ is Cohen-Macaulay, let $B=\cup_{i=1}^r B_i$ be the decomposition of the scheme-theoretic base locus of $f$ into its scheme-theoretic irreducible components, and let $\ee_i:=\mult_{B_i}X$. If $L$ is nef, then
$$
  \delta \geq \sum_i \ee_i \deg_{B_{i,\red}}L\,.
$$
In particular, if $B$ is non-empty and $L$ is ample, then the discrepancy $\delta$ is strictly positive.
\end{cor}
\begin{proof}
Fix an irreducible component $B_i$ of $B$ and denote by $k:=\dim X-\dim B_i$. We will single out an irreducible component $Z_i$ of $V^k$ such that $Z_{i,\red}=B_{i,\red}$ and
$$
\ee_i=\mult_{B_i}X=\mult_{Z_i}X=\ell(\calO_{X,Z_{i,\red}}/\mathcal{I}_{Z_i})\,.
$$
Since $B_i$ is contained in each divisor $D_j$, it follows that $B_i$ is contained in $R^{k-1}$, and, for dimensional reasons, $B_{i,\red}$ is the support of an irreducible component $Z_i$ of $R^{k-1}\cap D_k$. By construction of the $v$-cycles, this means that $Z_i$ is an irreducible component of $V^k$. In particular, we have  $Z_{i,\red}=B_{i,\red}$. Let $(R,\m)$ be the local ring of $X$ at the generic point of $Z_{i,\red}$. The ideal $\q$ defining $B_i$ in the local ring $R$ is generated by the image in $R$ of the linear system $W$. Recall that $s_j$ is the section defining $D_j$. The subscheme $Z_i$ is defined in $R$ by the image of the sections $s_1, \dots, s_k$. These sections vanish on $B_i$, so they lie in $\q$. The sections are generic elements of $W$, so we can apply Proposition~\ref{CAtheorem} to conclude that $\ell(\calO_{X,Z_{i,\red}}/\mathcal{I}_{Z_i})=\mult_{Z_i}X=\mult_{B_i}X$.

We now use the nefness of $L$, which implies that its degree is non-negative on any subscheme $Y\subset X$.  Thus Theorem~\ref{main} implies
$$
\delta=\sum_{j=1}^n\deg_{V^j}L \geq \sum_{i=1}^r \deg_{Z_i}L.
$$
where $Z_1,\dots , Z_r$ is the list of irreducible components of $v$-cycles associated to the irreducible components $B_1,\dots , B_r$ of the base locus by the procedure described above.
(Note that if $B$ is zero-dimensional we do not need to get rid of higher codimension $v$-cycles, so the inequality becomes an equality.)

The degree $\deg_{Z_i}L$ is an intersection number which can be computed in the Chow group of $X$, so $\deg_{Z_i}L=\ee_i\deg_{B_{i,red}}L$.
\end{proof}
Let us compute the cycles $Z_i$ of the previous proof in a particular example.
\begin{example}
Let $X$ be the singular quadric surface defined by $x^2+y^2+w^2=0$ in $\PP^3$, with coordinates $[x:y:w:z]$; let $p=[0:0:0:1]$ be the singular point. Let $L$ be $\calO_{\PP^3}(1)|_X$, and consider the linear system $W$ on $X$ generated by the restrictions of $x,y$ and $w$. The base locus is the point $p$ endowed with the reduced scheme structure. On the other hand, if we run Vogel's algorithm using as sections $s_1=x$, $s_2=y$ and $s_3=w$, the $v$-cycle $V^2$ is defined in $\PP^3$ by the ideal $(x^2+y^2+w^2, x,y)$; thus $V^2$ is supported at $p$ but is not reduced. This is a case where $Z$ and $B$ have the same support and the same multiplicity in $X$, but they are different as schemes.
\end{example}
\begin{rem}[Zero-dimensional base locus]\label{dim_zero}
Suppose $B=\cup_i B_i$, where each $B_i$ is a possibly non-reduced scheme supported on a closed point $p_i$. In this case, as explained in the proof of Corollary~\ref{bound_disc},
$$
\delta=\deg_{V^n}L=\sum \mult_{B_i}X\,.
$$
This formula has a down to earth proof. Take $n$ general effective Cartier divisors $D_i$ in $|W|$: they intersect in a point $q\in\PP W^{\vee}$; the intersection of the $D_i$ in $X$ consists of the points $p_i$ counted with multiplicity $\mult_{B_i}X$ and other $\deg_XL-\sum \mult_{B_i}X$ distinct reduced points $q_i$. These points $q_i$ are exactly the fiber of $f$ over $q$, proving the formula for $\delta$ in this case.
\end{rem}
The case we are interested in is when  $X=\Theta$ is the theta divisor of a ppav $(A,\Theta)$ of dimension $g$, and $f=\calG$ is the Gauss map. We thus obtain
$$
  \deg\calG=g!-\delta.
$$
%%%%%%%%%%%%%%%%%%%%%%%%%%%%%%%%%%%%%%%%%%%%%

\section{The multiplicity bound for theta divisors}\label{section:multi}
In this section we prove Theorem~\ref{thm:multbound} bounding the multiplicity of the theta divisor, by an analysis of the discrepancy associated to the Gauss map. It is well-known~\cite[Section 4.4]{BL} that for an indecomposable ppav the Gauss map is dominant and has positive degree. We start with the following observation.
\begin{lm}
For an indecomposable ppav the degree of the Gauss map $\deg \calG$ is even. If $g\geq 3$, then
$$
4\leq \deg \calG \leq g!\,.
$$
Moreover, $\deg \calG = g!$ if and only if the theta divisor is smooth.
\end{lm}
\begin{proof}
Let $\iota:z\mapsto -z$  be the involution of the ppav; since the Gauss map is invariant under $\iota$, its degree is even.

If the degree of the Gauss map for a ppav of dimension $g\ge 3$ were equal to $2$, then $\Theta/\iota$ would be rational. However, the quotient $\Theta/\iota$ cannot be rational because there exist non-zero holomorphic $\iota$-invariant $2$-forms on $\Theta$.

The inequality $\deg\calG \leq g!$ follows from Corollary~\ref{bound_disc} and the fact that the Gauss map is given by an ample line bundle. The last claim holds because the base locus of the Gauss map is exactly $\Sing(\Theta)$.
\end{proof}
Since the base locus of the Gauss map is equal to $\Sing(\Theta)$, the results of the previous section combined with this lemma yield
$$
\delta=g!-\deg\calG\leq g!-4\,,
$$
where $\delta$ is the discrepancy associated to $\calG$. To translate this into geometric information about the singularities of the theta divisor, we need another commutative algebra statement.
\begin{prop}\label{bound_Samuel}
Let $D\subset X$ be an irreducible divisor in an $n$-dimensional smooth variety $X$, and let $Z$ be a $d$-dimensional irreducible component of $\Sing(D)$. Let $m:=\mult_{Z_{red}} D$   and $\ee:=\mult_ZD$. Then
$$ \ee \geq m(m-1)^{n-d-1}.$$
\end{prop}

\begin{proof} Let $(S,\n)$ be the local ring of $Z_{red}\subset X$. It is a regular local ring of dimension $n-d$.    A local equation of $D$ is an element $f\in \n^m$.  Let   $(R,\m) = (S/(f),\n/(f))$ be the local ring of $Z_{red}\subset D$. It is a Cohen-Macaulay ring of dimension $n-d-1$. Since $m:=\mult_{Z_{red}} D$, we have
$$
  \ee(\m,R) = m\,.
$$
%\cite[Th. 14.9 p. 109]{matsumura}).
Let $\q:=(f_1,\dots, f_n)$ be the ideal of $Z\subset D$, where the $f_i$'s are the classes of the partial derivatives of $f$.  Note that $\q \subseteq \m^{m-1}$. Let  $s_1, \dots, s_{n-d-1}\in \q$ be a system of parameters.  Then
$$
  \ell(R/(s_1, \dots, s_{n-d-1})) \ge (m-1)^{n-d-1}\ee(\m,R) = (m-1)^{n-d-1}m
$$
by~\cite[Thm. 14.9, p. 109]{matsumura}.  Choosing $s_1, \dots, s_{n-d-1}$
to be a general linear combination of $f_1,\dots, f_n$ we can apply Proposition~\ref{CAtheorem}. Then we get
$$
  \ee(\q,R)= \ell(R/(s_1, \dots, s_{n-d-1}))\ge (m-1)^{n-d-1}m,
$$
proving the proposition.
\end{proof}
We can now prove our main result.
\begin{proof}[Proof of theorem~\ref{thm:multbound}]
The theorem follows by combining Corollary~\ref{bound_disc}, Proposition~\ref{bound_Samuel}, and the bound $\de\leq g!-4$.
\end{proof}

The estimates on the degree of the Gauss map and on the multiplicity of the singularities of Theorem~\ref{thm:multbound} can be improved using the vast literature on lower bounds of $\deg_{V_{i,\red}}\Theta$ for $d$-dimensional subvarieties $V\subseteq A$.  For any $p$ we will denote by $HA^{p,p}(A)$ the algebraic cohomology, that is the subgroup of $H^{p,p}(A,\CC)\cap H^{2p}(A,\ZZ)$ generated by the dual classes to the fundamental cycles of $(g-p)$-dimensional subvarieties of $A$ (so in particular $HA^{1,1}$ is the Neron-Severi group).

When $d=1$, a refinement of Matsusaka's criterion due to Ran states that if $V$ generates the abelian variety, then $\deg_V\Theta\ge g$, with equality if and only if $A$ is a Jacobian and $V$ is the Abel-Jacobi curve~\cite[Corollary 2.6 and Theorem 3]{Ran}. Combining this with  Theorem~\ref{thm:multbound} immediately yields the following:
\begin{cor}\label{app1}
For  a one-dimensional irreducible component $V$ of  $\Sing(\Theta)$, denote by $m:=\mult_{V_{red}}\Theta$. If $A$ is a simple abelian variety that is not a Jacobian, then
$$
 m(m-1)^{g-2}(g+1)\leq g!-\deg \calG     \le g! -4.
$$
\end{cor}
In general for $d=\dim V>1$, Ran ~\cite[Corollary II.6]{Ran} shows that
$$
  \deg_V\Theta \geq \binom{g}{d}
$$
if $V$ is non-degenerate in the sense of that paper.

If the class of $V$ in $HA^{g-d,g-d}(A)$ is an integer multiple of the so-called minimal class $\theta_{g-d}:=\frac{1}{(g-d)!}\Theta^{g-d}$ (which must be the case if $HA^{g-d,g-d}(A)=\ZZ$), then we have the bound
$$
  \deg_V\Theta\geq \frac{g!}{(g-d)!}\,,
$$
by noticing that $V$ must then be a positive integral multiple of the minimal class, since it has positive intersection number with $\Theta^d$, we get in particular the following
\begin{cor}\label{bound_hodge}
For a $d$-dimensional irreducible component $V$ of  $\Sing(\Theta)$, denote by $m:=\mult_{V_{red}}\Theta$. If  $HA^{g-d,g-d}(A)=\ZZ$, then
$$
 m(m-1)^{g-d-1}  \frac{g!}{(g-d)!}  \leq g!-\deg \calG     \le g! -4\,.
$$
\end{cor}
Note that the case of $d=g-2$ of the above corollary gives precisely corollary~\ref{smooth_cod_1}

\medskip
By applying Theorem~\ref{thm:multbound}, we can also obtain further results for the case of isolated singular points, extensively studied in the literature. For an isolated point $z$ of $\Sing(\Theta)$, let $B_z$ be the irreducible component of $\Sing(\Theta)$ supported at $z$, and let $\ee(z):=\mult_{B_z}\Theta$.  We first need the following proposition.
\begin{prop}\label{defect_ak}
If $z\in\Sing(\Theta)$ is an isolated double point such that the Hessian matrix of $\theta$ (second partial derivatives) at $z$ has rank at least $g-1$, then there exists a local coordinate system in which locally the theta function can be written as
$$
 \theta=x_1^2+\ldots+x_{g-1}^2+x_g^\ell
$$
for some $\ell\geq 2$. In this case $\ee(z)=\ell$.
\end{prop}
\begin{proof}
The existence of a local coordinate system where $\theta$ can be written as above follows immediately from the holomorphic Morse lemma~\cite[Thm.~2.26]{Zoladek}.

Let $s_i$ be as in Proposition~\ref{CAtheorem}; since we are interested just in the linear span of the $s_i$, applying Gauss elimination algorithm we can assume that $s_i=x_i$, for $i=1,\dots , g-2$ and $s_{g-1}=x_{g-1}+x_g^{\ell-1}$. Since
$$
\ee=\dim_{\mathbb{C}} S/(f, s_1,\dots , s_{n-d-1}),
$$
where $S$ is the local ring of the abelian variety at $z$, the statement follows.
\end{proof}
In genus 4 the proposition immediately gives the following corollary, which was obtained by completely different methods in~\cite{smva2}.
\begin{cor}
Let $(A,\Theta)\in\calA_4$ be a Jacobian of a smooth non-hyperelliptic curve with a semi-canonical $g_3^1$ (i.e.~with a theta-null). Then locally around the unique singular point $z$ of the theta divisor, the theta function can be written in an appropriate local coordinate system as $\theta=x_1^2+x_2^2+x_3^2+x_4^4$.
\end{cor}
\begin{proof}
Recall that the degree of the Gauss map for any non-hyperelliptic genus 4 Jacobian is $\binom{2\cdot 4-2}{4-1}=20=4!-4$, so the discrepancy $\de$ is equal to $4$.

By~\cite[p.~232]{acgh}, in this case $\Sing(\Theta)_{\red}=\lbrace p\rbrace$ and the rank of the Hessian of the theta function at $p$ is equal to $3$. The result now follows from Proposition~\ref{defect_ak} and Corollary~\ref{bound_disc}.
\end{proof}
\begin{rem}
Smith and Varley~\cite{smva2} prove that coordinates $x_i$ in the corollary above can furthermore be chosen to be equivariant, i.e.~such that they all change sign under the involution $z\mapsto -z$ of the ppav. This also follows from the above proof: indeed, in Proposition~\ref{defect_ak} note that if the point $z$ is two-torsion, i.e.~fixed by the involution, then since the theta divisor is also invariant under the involution, as a set, one can use the equivariant version of the holomorphic Morse lemma to ensure that the coordinates $x_i$ are equivariant under the involution as well.
\end{rem}
The proposition also immediately yields the following easy general bound, which seems not to have been previously known.
\begin{cor}
If the point $z$ is an isolated singular point of $\Theta$, and the rank of the Hessian of the theta function at $z$ is at least $g-1$, then, with the notation of Proposition~\ref{defect_ak}, we have $\ell\le g!-4$.
\end{cor}
In a similar spirit, we can  also prove the following bound
\begin{cor}\label{bound_points}
A theta divisor contains at most $\frac{1}{2}(g!-4)$ isolated singular points.
\end{cor}
\begin{proof}
For any isolated point $z\in\Sing(\Theta)$, we can apply the multiplicity one criterion, which states that a primary ideal in a local ring has Samuel multiplicity one if and only if the ring is regular and the ideal is the maximal ideal, to see that $\ee(z)\ge 2$; references for the criterion are~\cite[Proposition 7.2]{fultonbook} or~\cite{MultOne}. Then the corollary follows  from Theorem~\ref{thm:multbound}.
\end{proof}
Perhaps the next interesting case of ppav's with singular theta divisors beyond Jacobians of curves are intermediate Jacobians of smooth cubic threefolds. In this case the degree of the Gauss map can be computed from the geometric description given by Clemens and Griffiths~\cite{CG}, see~\cite[Remark 7 (b)]{Kramer}. Let us sketch the computation. In ~\cite{CG}, the authors consider the Fano surface $S$ of lines of a cubic 3-fold $F$ and a natural rational map:
\[
\Psi:S\times S \longrightarrow \Theta
\]
where $\Theta\subset J(F)$ is the theta divisor. They prove that $\Psi$ is generically $6$ to $1$,~\cite[page 348]{CG}. The composition $\calG\cdot \Psi$ associates to an ordered pair of lines $(\ell_1,\ell_2)\in S\times S$ the hyperplane  $H:= \langle\ell_1,\ell_2\rangle \subset \PP^4$. The degree of $\calG\cdot \Psi$ equals the number of ordered pairs of non-intersecting lines contained in the cubic surface $H \cap F$ for a general $H \in \PP^{4\vee}$. There are $27\times 16$ ordered pairs of non-intersecting lines on a nonsingular cubic surface; therefore the degree of $\calG$ is
\[
\deg(\calG) = \frac{\deg(\calG\circ\Psi)}{\deg(\Psi)}=\frac{27\times 16}{6} = 72
\]
We remark that since the degree of the Gauss map for Jacobians of genus 5 curves is either 16 (for hyperelliptic curves) or $\binom{8}{4}=70$ (for non-hyperelliptic curves), this computation already shows that the intermediate Jacobian of a smooth cubic threefold is not a Jacobian of a curve --- allowing one to bypass a longer argument for this in~\cite{CG}.

\smallskip
Our machinery gives a quick alternative computation for this degree.
\begin{prop}
The degree of the Gauss map of the intermediate Jacobian of a smooth cubic threefold is equal to 72.
\end{prop}
\begin{proof}
It is well-known~\cite{beauville} that $B=\Sing(\Theta)$ is supported on a single point $z$, which has multiplicity $3$; the tangent cone of that singularity is the cubic threefold. In this case, as explained in Remark~\ref{dim_zero}, the discrepancy $\delta$ equals the multiplicity $\ee$ of $\Theta$ along $B$.

We use Proposition~\ref{CAtheorem} to compute $\ee$. Locally near its singularity, the theta function is a cubic polynomial defining the cubic, plus higher order terms. Since the cubic is smooth, its partial derivatives are independent, and thus in the setup of Proposition~\ref{CAtheorem} we can choose $s_i$ to be the theta function (which vanishes at $z$ to order 3), and four of its partial derivatives (each of which vanishes to second order). Thus we compute the Samuel multiplicity of $z$, or equivalently $\deg_B\Theta$, to be equal to $3\cdot 2^4=48$. It follows that the discrepancy is equal to 48. Thus the degree of the Gauss map in this case is $\deg\calG=5!-48=72$.
\end{proof}

\section{The stratification by the degree of the Gauss map}
The moduli space of ppav's $\calA_g$ is an orbifold of complex dimension $g(g+1)/2$. The classical Andreotti-Mayer loci $N_k^{(g)}$ are defined to be the loci of ppav such that $\dim\Sing(\Theta)\ge k$. These loci are closed. In analogy with the Andreotti-Mayer loci, we define the Gauss loci.
\begin{df}[Gauss loci]
For any $d\in\ZZ_{\ge 0}$ we let $G_d^{(g)}\subseteq\calA_g$ be the set of all ppav's for which the degree of the Gauss map is less than or equal to $d$.
\end{df}
For a decomposable ppav, the Gauss map is not dominant, so the degree of the Gauss map is $0$; in other words $\calA_g^{\op{dec}}\subset G_0^{(g)}$, where we denoted by $\calA_g^{\op{dec}}$ the locus of decomposable ppav's. By definition we have $G_0^{(g)}\subseteq G_1^{(g)}\subseteq\ldots\subseteq G_{g!}^{(g)}=\calA_g$. Since $\deg\calG$ is always even, $G_{2d}^{(g)}=G_{2d+1}^{(g)}$ for any $d$. For $g\geq 3$, since the degree of the Gauss map on an indecomposable ppav is at least 4, $G_0^{(g)}=G_1^{(g)}=G_2^{(g)}=G_3^{(g)}=\calA_g^{\op{dec}}$.

\medskip
We now discuss in detail the Gauss loci in low genus. In genus $g=2$ the theta divisor of any indecomposable principally polarised abelian surface is smooth, and the degree of the Gauss map is equal to two. To summarise,
$$
 G_0^{(2)}=\calA_2^{\op{dec}};\qquad G_2^{(2)}=\calA_2.
$$

\smallskip
In genus $3$ the theta divisor of any non-hyperelliptic Jacobian is smooth, while the theta divisor of a hyperelliptic Jacobian has a unique singular point, which is an ordinary double point. Thus
$$
 G_0^{(3)}=G_2^{(3)}=\calA_3^{\op{dec}};\qquad G_4^{(3)}=\overline{\calH}_3;\qquad G_6^{(3)}=\calA_3,
$$
where $\calH_3$ denotes the locus of hyperelliptic Jacobians, and we note that the boundary of $\calH_3$ in $\calA_3$ is in fact equal to the decomposable locus.
We note that in this case the Gauss loci are equal to the Andreotti-Mayer loci, but already in genus $4$ this is not the case.

\smallskip
We now turn to genus 4. Recall that $N_1^{(4)}=\overline\calH_4$, and the degree of the Gauss map for a Jacobian of a smooth hyperelliptic genus 4 curve is equal to $2^{4-1}=8$ --- this is to say $\calH_4\subset G_8^{(4)}\setminus G_6^{(4)}$. We also recall from~\cite{beauvillegenus4} that $N_0^{(4)}=\calJ_4\cup\theta_{\rm null}$, where $\calJ_4$ denotes the locus of Jacobians, and $\theta_{\rm null}$ denotes the theta-null divisor --- the locus of ppav's with a singular 2-torsion point on the theta divisor. Recall also that the degree of the Gauss map for any non-hyperelliptic genus 4 Jacobian is equal to 20. Furthermore, as is well-known, and as also immediately follows from our bound for the discrepancy, an indecomposable 4-dimensional ppav cannot have an isolated singular point $z$ on the theta divisor with $m=\mult_z\Theta>2$ (since $3\cdot 2^3=24>20=4!-4$). Finally,  in~\cite{grsmJac4} it was shown that if the theta divisor of a 4-dimensional indecomposable ppav has a double point that is not ordinary, it is the Jacobian of a curve. To summarize (and further using~\cite{beauvillegenus4}), if a 4-dimensional indecomposable ppav that is not a Jacobian has a singular theta divisor, the only singularities of it are at two-torsion points (vanishing theta-nulls), and each of those is an ordinary double point.

We can now describe the Gauss loci in genus 4.
\begin{prop}\label{prop:genus4close}
\begin{enumerate}
\item[(1)] All the Gauss loci $G_{2d}^{(4)}$ are closed.
\item[(2)] All the Gauss loci $G_{2d}^{(4)}$, for $d=2,\ldots,12$ are distinct.
\item[(3)] The locus $G_4^{(4)}$ is equal to the union of the locus of decomposable ppav's $\calA_4^{\op{dec}}=G_0^{(4)}$ (which has two irreducible components $\calA_1\times\calA_3$, of dimension 7, and $\calA_2\times\calA_2$ of dimension 6), and one point $\lbrace A_V\rbrace$, the fourfold discovered by Varley in~\cite{varleyweddle}.
\item[(4)] The 9-dimensional locus $\calJ_4$ is an irreducible component of $G_{20}^{(4)}$, which has another 8-dimensional irreducible component.
\item[(5)] The 7-dimensional locus $\calH_4$ is an irreducible component of $G_8^{(4)}$, which has another 2-dimensional irreducible component.
\end{enumerate}
\end{prop}
\begin{proof}
Indeed,  the above results imply that for a 4-dimensional indecomposable ppav that is not a Jacobian, the only possible singularities of the theta divisor are theta-nulls that are ordinary double points, each of which contributes precisely two to the discrepancy by Proposition~\ref{defect_ak}. The degree of the Gauss map is then $24-2k$, where $k$ is the number of vanishing theta constants. Since the vanishing of an even theta constant defines a divisor on a level cover $\calA_4(4,8)$ of $\calA_4$, the locus where a given collection of theta constants vanishes is always closed in $\calA_4(4,8)$, and thus its image is closed in $\calA_4$. Thus all the Gauss loci are closed within $\calA_4\setminus N_1^{(4)}$. Since $N_1^{(4)}=\overline\calH_4\cup\calA_4^{\op{dec}}\subset G_8^{(4)}$, it follows that $G_{2d}^{(4)}$ are closed for $d\ge 4$, and to prove (1) it remains to show that $G_4^{(4)}$ and $G_6^{(4)}$ are closed. The statement about $G_4^{(4)}$  follows from (3), which we will prove shortly. To prove that  $G_6^{(4)}$ is closed, all we need to show is that a family in $G_6^{(4)}$ can not degenerate to a smooth hyperelliptic Jacobian --- which lies in $G_8^{(4)}\setminus G_6^{(4)}$. We will do this once we are done with the rest of the proof.

To show that all Gauss loci are distinct, recall that
Varley in~\cite{varleyweddle} constructed an indecomposable principally polarised abelian fourfold $A_V$, which is not a Jacobian, but has 10 vanishing even theta-nulls (and note that the above discussion reproves that a 4-dimensional indecomposable ppav can have at most $k=10$ vanishing theta constants). Thus the degree of the Gauss map for $A_V$ is equal to $4$, so that  $A_V\in G_4^{(4)}$. Debarre~\cite{Debarreannulationdimension4} showed that $A_V$  is the unique indecomposable ppav in $\calA_4$ such that the singular locus of the theta divisor consists of $10$ isolated points; in our language, this implies (3), which also shows that the locus $G_4^{(4)}$ is closed. Thus locally on $\calA_4(4,8)$ the point $A_V$ is defined by the vanishing of the corresponding 10 theta constants. The vanishing locus of each theta constant is a hypersurface in $\calA_4(4,8)$, and since all 10 of them locally intersect in expected codimension 10, any subset also locally intersects in expected codimension. Thus for any $k\le 10$, there is a $(10-k)$-dimensional locus of ppav's near $A_V$ with exactly $k$ vanishing theta constants (and which do not lie on $\calJ_4$, since $A_V$ is not a Jacobian). Thus locally near $A_V$ there exist ppav's with the degree of the Gauss map equal to $24-2k$, for any $k\le 10$.

Furthermore, by the above we see that there exists an irreducible $(10-k)$-dimensional component of $G_{24-2k}^{(4)}$ containing $A_V$. Taking $k=2$ and $k=8$ shows the existence of irreducible components of $G_{20}^{(4)}\setminus\overline\calJ_4$ and of $G_8^{(4)}\setminus\overline\calH_4$ of the claimed dimensions. Furthermore, we note that the boundary of $\calH_4$ within $\calA_4$ is equal to
$$
(\overline\calH_2\times\overline\calH_2)\cup(\overline\calH_1\times\overline\calH_3)=(\calA_2\times\calA_2)\cup (\calA_1\times\calH_3)\subsetneq\calA_4^{\op{dec}}=(\calA_2\times\calA_2)\cup(\calA_1\times\calA_3).
$$
Recalling that $\overline\calJ_4$ and $\overline\calH_4$ are irreducible, and that the Gauss map on these loci generically has degree 20 and 8, respectively, proves parts 4 and 5.

Finally, we return to statement (1). Suppose we have a family $A_t$ of ppav's contained in $G_6^{(4)}$, degenerating to some $A_0\in\overline\calH_4$. We want to prove that $A_0$ cannot be a smooth hyperelliptic Jacobian --- which is equivalent to showing that $A_0\in\calA_4^{\op{dec}}$. Indeed, a generic $A_t$ must be a ppav that has at least 9 vanishing theta constants, and thus the same configuration of 9 theta constants must also vanish on $A_0$. It is known that for any smooth hyperelliptic Jacobian there are exactly 10 vanishing even theta constants, every three of them forming an azygetic triple (this results seems to go back to Riemann originally, we refer to \cite{mumf2} for a detailed discussion). Thus if among the 9 vanishing theta constants on $A_t$ there exists a syzygetic triple, then a syzygetic triple of theta constants vanishes at $A_0$, and so $A_0$ cannot be a smooth hyperelliptic Jacobian. On the other hand, if an azygetic 9-tuple of theta constants vanishes on $A_t$, then $A_t$ is a hyperelliptic Jacobian or a product, by~\cite[Lemma 6]{Igusa}. But then since $A_t\in G_6^{(4)}$, it follows that $A_t\in\calA_4^{\op{dec}}$.
\end{proof}
In genus 5, if the theta divisor is smooth, the degree of the Gauss map is equal to 120. For intermediate Jacobians of cubic threefolds the degree of the Gauss map is $72$, for the Jacobian of any non-hyperelliptic curve it is equal to $70$, while for the Jacobian of any hyperelliptic curve it is equal to $16$. In this case our results show that if all singular points of theta divisor are isolated (i.e.~on $N_0^{(5)}\setminus N_1^{(5)}$), there are at most $(5!-4)/2=58$ such singular points --- which in particular is much better than the conjectured bound of $2^4(2^5+1)-3^5=285$ for the number of vanishing even theta constants in this case (see~\cite{aufpism} for the recent results on this conjecture).

All irreducible components of the Andreotti-Mayer locus $N_1^{(5)}$ were described explicitly by Debarre~\cite[Prop.~8.2]{debarreAMloci} via the Prym construction, using the fact that the Prym map is dominant onto $\calA_5$. While it appears possible to compute the degree of the Gauss map for a generic point of every irreducible component of $N_1^{(5)}$, determining the degree of the Gauss map for every point of $N_1^{(5)}$ would require computing the degree of the Gauss map for Pryms of all admissible covers that arise as degenerations, which involves a large number of cases, and which we thus do not undertake. Furthermore, we have $G_{16}^{(5)}\supset \calH_5$, and Debarre showed that  $N_2^{(5)}\setminus\calA_5^{\op{dec}}=\calH_5$~\cite[Section 1]{cilvdg1}, while the results of Ein and Lazarsfeld~\cite{EinLaz} are that $N_3^{(5)}=\calA_5^{\op{dec}}$, which is thus equal to $G_4^{(5)}$. We do not know whether various Gauss loci are distinct.

\smallskip
In higher genus, other than the classical case of Jacobians, the degree of the Gauss map is known for a generic Prym: Verra~\cite{Verra} showed that it is equal to $D(g)+2^{g-3}$, where $D(g)$ is the degree of the variety of all quadrics of rank at most $3$ in $\PP^{g-1}$. We note that for $g=5$ a generic ppav is a Prym, with a smooth theta divisor, and thus the result of Verra recovers the degree $5!=120$ of the Gauss map in this case.

\end{document}